\documentclass[11pt,a4paper]{amsart}
\usepackage[english]{babel}
\usepackage{amsthm,amsmath,amssymb}
\usepackage{hyperref}
\usepackage{color}
\usepackage{a4wide}

\newtheorem{thm}{Theorem}[section]
\newtheorem{cor}[thm]{Corollary}
\newtheorem{lem}[thm]{Lemma}
\newtheorem{prop}[thm]{Proposition}

\newtheorem{rem}[thm]{Remark}
\numberwithin{equation}{section}

\newcommand{\rd}{{\rm d}}
\newcommand{\vep}{\varepsilon}
\newcommand{\dd}{\delta}

\def\supp{\mathrm{supp}\,} 

\newcommand{\RR}{\mathbb{R}}
\newcommand{\NN}{\mathbb{N}}

\newcommand{\ve}{\varepsilon}

\def\dist{\mathrm{dist}} 

\begin{document}

\title[Convergence in relative error for the PME in a tube]{Convergence in relative error for the Porous Medium equation in a tube}

\author[A.~Audrito]{Alessandro Audrito}
\address[A.~Audrito]{Department of Mathematics, ETH Zurich, R\"amistrasse 101, 8092 Zurich, Switzerland.
}
\email[]{alessandro.audrito@math.ethz.ch}

\author[A.~Gárriz]{Alejandro Gárriz}
\address[A.~Gárriz]{Institut de Mathématiques de Toulouse. 1R3, Université Paul Sabatier, 118 Rte de Narbonne, 31400 Toulouse, France}
\email[]{alejandro.garriz@math.univ-toulouse.fr}

\author[F.~Quir\'{o}s]{Fernando Quir\'{o}s}
\address[F.~Quir\'{o}s]{Departamento de Matem\'aticas, Universidad Aut\'onoma de Madrid, 28049 Madrid, Spain\\
and\\
Instituto de Ciencias Matem\'aticas ICMAT (CSIC-UAM-UCM-UC3M), 28049 Madrid, Spain}
\email[]{fernando.quiros@uam.es}

\date{}



\maketitle

\begin{abstract}
Given a bounded domain $D \subset \RR^N$ and $m > 1$, we study the long-time behaviour of solutions to the Porous Medium equation (PME) posed in a tube
\[
\partial_tu = \Delta u^m \quad \text{ in } D \times \RR, \quad t > 0,
\]
with homogeneous Dirichlet boundary conditions on the boundary $\partial D \times \RR$ and suitable initial datum at $t=0$. In two previous works, V\'azquez and Gilding \& Goncerzewicz proved that a wide class of solutions exhibit a traveling wave behaviour, when computed at a logarithmic time-scale and suitably renormalized. In this paper, we show that, for large times, solutions converge in \emph{relative error} to the Friendly Giant, i.e., the unique nonnegative solution to the PME posed in the section $D$ of the tube  (with homogeneous Dirichlet boundary conditions) having a special self-similar form. In addition, \emph{sharp} rates of convergence and \emph{uniform} bounds for the location of the \emph{free boundary} of solutions are given.
\end{abstract}


%
%
%
%
%
%
%
%
%
%
\section{Introduction}\label{SectionIntroduction}
We investigate the long-time behaviour of nonnegative solutions to the Cauchy-Dirichlet problem for the Porous Medium equation (PME)
\begin{equation}\label{eq:PLETUBULARDOMAIN}
\begin{cases}
\partial_tu = \Delta u^m \quad &\text{ in } \Omega \times (0,\infty), \\
u = 0 \quad &\text{ in } \partial\Omega \times (0,\infty), \\
u(x,0) = u_0(x) \quad &\text{ in } \Omega,
\end{cases}
\end{equation}
posed in a \emph{tubular} domain $\Omega \subset \RR^{N+1}$:
\begin{equation*}\label{eq:TUBULARDOMAINASSUMPTION}
\Omega = D \times \RR,
\end{equation*}
where $D \subset \RR^N$ is a bounded domain with smooth boundary, which is assumed without loss of generality to contain the origin $O$, and the diffusion parameter satisfies $m > 1$ (degenerate diffusion framework). The initial data are assumed to be nontrivial and to satisfy
\begin{equation}\label{eq:AssInitialData}
u_0 \geq 0 \quad \text{in } \Omega, \qquad u_0 \in C_0(\overline{\Omega}), \qquad u_0^m \in C^1(\overline{\Omega}).
\end{equation}
The most peculiar aspect of this setting is the \emph{cylindrical} shape of the spatial domain $\Omega$, that we call \emph{tube}. When the PME equation is posed either in a bounded domain or in the whole space, the long-time behaviour of nonnegative solutions has been widely investigated and the theory has been fully developed, see for instance \cite{AronsonPeletier1985:art,Vazquez2002:art,Vazquez2004:art,V2:book,S-V1:art}. In both frameworks, if the initial data belong to a reasonable class, the corresponding solutions converge to some \emph{special solutions} to the PME for large times: the \emph{Friendly Giant} in the bounded domain setting (see \eqref{eq:FriendlyGiant} and \cite{AronsonPeletier1985:art,Vazquez2004:art,S-V1:art}), while to the \emph{Barenblatt solution} when the equation is posed in the whole $\RR^N$ (see \cite{Vazquez2002:art,V2:book}).

Even though the tubular setting is intermediate between the above two and presents significant novelties in the description of the long-time behaviour, it  has been less studied: to the best of our knowledge, the only papers on the topic are \cite{Vazquez2004:art,Vazquez2007:art} by V\'azquez and \cite{GidingGoncerzewicz2016:art} by Gilding and Goncerzewicz (see also \cite{AA-JLV3:art} in the $p$-Laplacian diffusion setting). Nonnegative solutions to \eqref{eq:PLETUBULARDOMAIN} exhibit a \emph{traveling wave} (TW) behaviour as $t \to +\infty$, when computed at the correct rescaled variables (see \cite[Theorems 4.1 and 4.2]{GidingGoncerzewicz2016:art}, our main Theorem \ref{thm:LongTimeBeh} and the change of variables \eqref{eq:SECONDRENORMALIZATIONFORMULA}). In this paper we establish new results in this direction and we refine some of the previous ones. As we will explain later, our main theorem is \emph{sharp}.

TW fronts arise frequently in the study of reaction-diffusion equations (see for instance \cite{Fisher:art,K-P-P:art,Aro-Wein1:art,B2:art} and \cite{DePablo-Vazquez1:art,DP-S:art,Biro2002:art,GKbook,AA-JLV:art,AA-JLV2:art,AA2017:art,Du-Quiros-Zhou:art,Garriz2018:art,Du-Garriz-Quiros:preprint} for nonlinear diffusion models) and are not expected in purely diffusive models. The reason why TW solutions appear in our context is the geometry of the domain: when $\Omega$ is a tube, solutions spread along the longitudinal direction, with a loss of mass through the fixed boundary $\partial \Omega$ which, in turn, is compensated by a linear reaction term appearing in the rescaled problem. As a consequence, the asymptotic behaviour for $t \to +\infty$ is in fact a sort of combination of two modes: in the middle of the tube solutions behave like the \emph{Friendly Giant} corresponding to the section $D \subset \RR^N$, while far away they look like a TW and their front moves with constant speed for large times.

Before entering into details, we recall that \emph{degenerate nonlinear} diffusion plays a crucial role in the entire theory (as explained in \cite[Section 7]{Vazquez2007:art} solutions in the range $m \leq 1$ behave quite different). In particular, when $m > 1$ solutions $u$ to \eqref{eq:PLETUBULARDOMAIN} (with compactly supported initial data) have a \emph{free boundary}, i.e., for every $t > 0$, $u(\cdot,t)$ is identically zero outside a set having compact closure in $\overline{\Omega}$ and the set $\{u(\cdot,t) > 0\}$ covers eventually the whole $\Omega$ when $t \to +\infty$. The study of the geometry of such set and its boundary is one of the most interesting (and complicated!) issues of the theory. As proved in \cite{Vazquez2007:art,GidingGoncerzewicz2016:art}, when $t$ is large enough, the \emph{free boundary} of solutions to \eqref{eq:PLETUBULARDOMAIN} is made of two sets described by two bounded and continuous functions $y = \Gamma_u^\pm(z,t)$ (we set $x=(z,y)$ with $z\in D$ and $y\in\RR$ for points $x\in \Omega$), moving towards both ends of the tube with logarithmic law:
\[
\Gamma_u^\pm(z,t) \sim \pm c_{\ast} \ln t \quad \text{ for } t \sim +\infty,
\]
where $c_{\ast} = c_{\ast}(m,D) > 0$ is the critical wave speed (cf. \cite[Theorem 3.1]{Vazquez2007:art} and/or Theorem \ref{thm:ExTWsVaz}). As mentioned above, this pattern strongly deviates from the asymptotics of the PME posed either in bounded domains or the whole space. In our main theorem, we will give a new proof of the logarithmic law for the \emph{free boundary}.

In the remaining part of the introduction we recall the most important results of \cite{Vazquez2004:art,Vazquez2007:art} and \cite{GidingGoncerzewicz2016:art} and we state our main theorem. According to the main ideas of \cite{Vazquez2007:art}, it is convenient to rescale and renormalize the solution $u$ to \eqref{eq:PLETUBULARDOMAIN} as explained in the following paragraph.

\subsection*{The transformed problem.} As just mentioned, the long-time behaviour can be equivalently described in terms of the rescaled and renormalized solution
\begin{equation}\label{eq:SECONDRENORMALIZATIONFORMULA}
v(x,\tau) = (t+t_0)^{\frac{1}{m-1}}u(x,t), \qquad \tau = \ln (t+t_0),
\end{equation}
where $u$ is a weak solution to problem \eqref{eq:PLETUBULARDOMAIN} and $t_0 > 0$ is fixed. At the initial time $t=0$, we have
\begin{equation}\label{eq:InDataTrasProb}
v(x,\tau_0) = t_0^{\frac{1}{m-1}}u_0(x) := v_0(x), \qquad \tau_0 := \ln t_0.
\end{equation}
It is not difficult to check that $v$ is a weak solution to the \emph{reaction-diffusion} problem
\begin{equation}\label{eq:REACTIONTRANSFORMATION}
\begin{cases}
\partial_{\tau} v = \Delta v^m + \frac{v}{m-1}   \quad &\text{ in } \Omega \times (\tau_0,\infty), \\
v = 0                                            \quad &\text{ in } \partial\Omega \times (\tau_0,\infty), \\
v(x,\tau_0) = v_0(x)                             \quad &\text{ in } \Omega,
\end{cases}
\end{equation}
and thus problems \eqref{eq:PLETUBULARDOMAIN} and \eqref{eq:REACTIONTRANSFORMATION} are equivalent, and we can pass from one to the other by means of the change of variables \eqref{eq:SECONDRENORMALIZATIONFORMULA}. Notice that, when $t_0 = 0$, $v$ is defined for all times $\tau \in \RR$ and satisfies
\begin{equation}\label{eq:REACTIONTRANSFORMATIONEternal}
\begin{cases}
\partial_{\tau} v = \Delta v^m + \frac{v}{m-1}   \quad &\text{ in } \Omega \times \RR, \\
v = 0                                            \quad &\text{ in } \partial\Omega \times \RR. \\
\end{cases}
\end{equation}
Both problems play a role in the analysis: the first is crucial in order not to loose the initial condition, while the second is used to construct the special wave solutions (see \cite{Vazquez2007:art}).
\subsection*{The work of Aronson and Peletier.} As anticipated in the above paragraphs, the long-time behaviour of solutions to \eqref{eq:PLETUBULARDOMAIN} has strong connections with the PME posed in bounded domains, with homogeneous Dirichlet boundary conditions, namely
\begin{equation}\label{eq:PLEBOUNDEDDOMAIN}
\begin{cases}
\partial_tu = \Delta u^m \quad &\text{ in } D \times (0,\infty), \\
u = 0 \quad &\text{ in } \partial D \times (0,\infty), \\
u(z,0) = u_0(z) \quad &\text{ in } D,
\end{cases}
\end{equation}
where $D \subset \RR^N$ is a bounded domain with smooth boundary and $u_0$ belongs to a suitable class of nontrivial initial data. Rescaling as in \eqref{eq:SECONDRENORMALIZATIONFORMULA}--\eqref{eq:InDataTrasProb}, we obtain a solution $v$ of \eqref{eq:REACTIONTRANSFORMATION} posed in $D$
\begin{equation}\label{eq:REACTIONTRANSFORMATIONBoundD}
\begin{cases}
\partial_{\tau} v = \Delta v^m + \frac{v}{m-1}   \quad &\text{ in } D \times (\tau_0,\infty), \\
v = 0                                            \quad &\text{ in } \partial D \times (\tau_0,\infty), \\
v(x,\tau_0) = v_0(x)                             \quad &\text{ in } D.
\end{cases}
\end{equation}
We rephrase below the main result of the seminal paper \cite{AronsonPeletier1985:art} by Aronson and Peletier.
\begin{thm}(\cite[Proposition 1, Theorems 1 and 2]{AronsonPeletier1985:art})\label{ASYMPTOTICSINBOUNDEDDOMAINS}
Let $v = v(z,\tau)$ be the nonnegative weak solution to problem \eqref{eq:REACTIONTRANSFORMATIONBoundD} with nontrivial and nonnegative initial data $v_0$, satisfying $v_0 \in C_0(\overline{D})$, $v_0^m \in C^1(\overline{D})$. Then
\begin{equation}\label{eq:RelErrVBdd}
\lim_{\tau \to \infty} \sup_{z \in D}  \;\left| \frac{v(z,\tau)}{\Phi(z)} - 1 \right| = 0,
\end{equation}
where $\Phi = \Phi(z)$ is the unique nonnegative (and nontrivial) weak solution to the stationary problem
\begin{equation}\label{eq:STATIONARYPROBLEM}
\begin{cases}
-\Delta \Phi^m = \tfrac{1}{m-1}\Phi \quad &\text{ in } D, \\
\Phi = 0              \quad &\text{ in } \partial D.
\end{cases}
\end{equation}
Furthermore,
\begin{equation}\label{eq:SSPosHopf}
\Phi \in C^\infty(D) \cap C(\overline{D}),  \qquad \quad    \Phi > 0 \quad \text{in } D,  \qquad \quad  \frac{\partial \Phi}{\partial \nu} < 0 \quad \text{in } \partial D,
\end{equation}
where $\nu = \nu(z)$ denotes the outward unit normal to $\partial D$ at $z$.
\end{thm}
Notice that, in terms of the solution $u$ to \eqref{eq:PLEBOUNDEDDOMAIN} with nontrivial and nonnegative initial data $u_0$ satisfying $u_0 \in C_0(\overline{D})$, $v_0^m \in C^1(\overline{D})$, the above statement asserts that
\begin{equation}\label{eq:RelErrUBdd}
\lim_{t \to \infty} \sup_{z \in D}  \;\left| \frac{u(z,t)}{U(z,t)} - 1 \right| = 0,
\end{equation}
where $U = U(z,t)$ is the unique weak solution to
\begin{equation}\label{eq:PLEBOUNDEDDOMAINSEPVARIABLE}
\begin{cases}
\partial_tU = \Delta U^m \quad &\text{ in } D \times (0,\infty), \\
U = 0 \quad &\text{ in } \partial D \times (0,\infty),
\end{cases}
\end{equation}
in self-similar form
\begin{equation}\label{eq:FriendlyGiant}
U(z,t) = t^{-\frac{1}{m-1}} \,\Phi(z),
\end{equation}
called the \emph{Friendly Giant}\footnote{This name comes from the fact that $U$ is a global upper bound for solutions to \eqref{eq:PLEBOUNDEDDOMAIN}, since $U(\cdot,t) \to +\infty$ in $D$ as $t \to 0^+$.}. The limits in \eqref{eq:RelErrVBdd} and \eqref{eq:RelErrUBdd} show that $u$ converges to $U$ in \emph{relative error} in $D$ for large times: this means that not only $u(\cdot,t) \to U(\cdot,t)$ uniformly in $\overline{D}$ as $t \to +\infty$, but also that $u$ and $U$ are \emph{comparable} up to the boundary $\partial D$, a highly nontrivial information since both $u$ and $U$ are zero in $\partial D$.

Keeping this fact in mind, we proceed by presenting the main result of V\'azquez, about the existence of a wave solution in the tubular domain framework.
\subsection*{The work of V\'azquez.} The starting point of our study is the paper \cite{Vazquez2007:art} (see also \cite{Vazquez2004:art}), where V\'azquez proved the existence of a wave solution to problem \eqref{eq:REACTIONTRANSFORMATION} and established a first result concerning the asymptotic behaviour for large times of solutions to \eqref{eq:PLETUBULARDOMAIN}. We rephrase below the main result in \cite{Vazquez2004:art,Vazquez2007:art}:
\begin{thm}\label{thm:ExTWsVaz} (\cite[Theorems 3.1 and 5.4]{Vazquez2007:art} and \cite[Theorem 3.2]{Vazquez2004:art})
There exists a unique $c_{\ast} > 0$ depending only on $m$ and $D$ such that problem \eqref{eq:REACTIONTRANSFORMATIONEternal} has a continuous\footnote{Actually, $\varphi$ is $C^\alpha(\overline{\Omega})$ for some $\alpha \in (0,1)$. See for instance \cite{AronsonPeletier1985:art,V2:book}.} nonnegative weak solution with wave form
\begin{equation}\label{eq:AnsatzTW}
\tilde{v}(z,y,\tau) = \varphi(z,y - c_{\ast}\tau),
\end{equation}
and the wave profile $\varphi = \varphi(z,\xi)$ satisfies
\begin{equation}\label{eq:Propphi}
\partial_\xi \varphi \leq 0,   \qquad    \lim_{\xi \to -\infty} \, \sup_{z\in D} \, \frac{\varphi(z,\xi)}{\Phi(z)} = 1,   \qquad   \varphi(z,\xi) = 0 \quad \text{for all }z\in D,\ \xi \geq \xi_0,
\end{equation}
for some $\xi_0 \in \RR$ (depending only on $m$ and $D$), where $\Phi = \Phi(z)$ is the unique nonnegative weak solution to problem \eqref{eq:STATIONARYPROBLEM}. Furthermore, for every compact set $\Omega' \subset \overline{\Omega}$, the solution $v$ to \eqref{eq:InDataTrasProb}--\eqref{eq:REACTIONTRANSFORMATION} satisfies
\[
\lim_{\tau \to +\infty} \sup_{(z,y) \in \Omega'} |v(z,y,\tau) - \Phi(z)| = 0,
\]
that is, in compact subsets of the tube, $v$ forgets about the longitudinal direction and behaves (in absolute value) as the solution to the PME posed in $D$.
\end{thm}
Some comments are now in order. First, notice that plugging the ansatz \eqref{eq:AnsatzTW} into \eqref{eq:REACTIONTRANSFORMATIONEternal}, we easily see that the pair $(c_\ast,\varphi)$ is a weak solution to the elliptic problem
\begin{equation}\label{eq:WaveEqcast}
\begin{cases}
\Delta \varphi^m + c_\ast\partial_\xi \varphi + \frac{\varphi}{m-1} = 0  \quad &\text{ in } \Omega, \\
\varphi = 0 \quad &\text{ in } \partial\Omega.
\end{cases}
\end{equation}
It thus follows that $\varphi$ is invariant under shifts along the longitudinal coordinate. Further, the reflection $\xi \to -\xi$ gives a profile $\tilde{\varphi}$ with speed $-c_\ast$, that is, the new profile travels towards the bottom end of the tube, and satisfies
\begin{equation}\label{eq:PropPhiReflected}
\partial_\xi \tilde{\varphi} \geq 0,   \qquad    \lim_{\xi \to +\infty} \, \sup_{z\in D} \, \frac{\tilde{\varphi}(z,\xi)}{\Phi(z)} = 1,   \qquad   \tilde{\varphi}(z,\xi) = 0 \quad \text{for all }z\in D,\ \xi \leq -\xi_0.
\end{equation}
Second, by the last relation in \eqref{eq:Propphi}, we have that the support of $\varphi$ is contained in $D\times(-\infty,\xi_0] \subset \Omega$ and so $\varphi$ has a free boundary, which is a locally H\"older hypersurface in $\Omega$ far from $\partial\Omega$ (see for instance \cite{Caff-Fried:art,Caf1980:art,CaffVazWol1987:art,Koch1990:art,KienzlerVazquezKoch2018:art}). It can be parametrized setting
\begin{equation}\label{eq:DefFBTW}
\xi = \Gamma_\varphi^+(z) := \inf\{\eta: \varphi(z,\eta) = 0\}, \quad z \in D.
\end{equation}
By \cite[Proposition 4.1]{Vazquez2007:art}, it turns out that the function $\Gamma_\varphi^+$ is bounded and continuous in $D$. Finally, notice that coming back to the usual variables, we obtain a family of special solutions to the PME (posed in $\Omega$ with homogeneous Dirichlet boundary conditions) defined by
\[
\psi_\ell(z,y,t) := t^{-\frac{1}{m-1}}\varphi(z, y - c_\ast \ln t - \ell), \qquad \ell \in \RR, \; t > 0.
\]
\subsection*{The work of Gilding and Goncerzewicz.} Later, in \cite{GidingGoncerzewicz2016:art}, Gilding and Goncerzewicz investigated the long-time behaviour of solutions to \eqref{eq:PLETUBULARDOMAIN} and further properties of the wave solutions introduced above, such as uniqueness and stability. Their arguments are based on a powerful \emph{invariance principle} (see \cite[Theorem 3.1]{GidingGoncerzewicz2016:art}) that allows to characterize the critical speed of propagation. Exploiting the same invariance principle, Gilding and Goncerzewicz showed that, in an appropriate scale, the solution to \eqref{eq:PLETUBULARDOMAIN} converges to a suitable displacement (in dependence of the initial data) of the wave solution \eqref{eq:AnsatzTW} as $t \to +\infty$. We present below the main results in \cite{GidingGoncerzewicz2016:art}, rephrasing them coherently with our notations.
\begin{thm}(\cite[Section 4]{GidingGoncerzewicz2016:art})\label{thm:GildGoncRecap} Let $v$ be the solution to \eqref{eq:InDataTrasProb}--\eqref{eq:REACTIONTRANSFORMATION} with initial datum satisfying \eqref{eq:AssInitialData} and let $(c_\ast,\varphi)$ the wave solution satisfying \eqref{eq:Propphi}--\eqref{eq:WaveEqcast}. Then the following assertions hold true:

\noindent{\rm (i)} Up to longitudinal shifts, the pair $(c_\ast,\varphi)$ is the unique solution to problem \eqref{eq:WaveEqcast} satisfying \eqref{eq:Propphi} and, furthermore,
\[
c_\ast = \frac{1}{(m-1)\sqrt{\lambda_1(D)}},
\]
where $\lambda_1(D) > 0$ is the first eigenvalue of the Dirichlet Laplacian in $D$.

\noindent{\rm (ii)} There exists $\ell_0 \in \RR$ depending only on $m$, $D$ and the initial data such that for every $\alpha\in\RR$, there holds
\begin{equation}\label{eq:ConvGilGonc}
\lim_{\tau \to +\infty} \sup_{z \in D, |y| \geq \alpha +  c_\ast\tau } |v(z,y,\tau) - \varphi(z,y - c_\ast\tau -\ell_0)| = 0.
\end{equation}
The shift $\ell_0$ is given explicitly in \cite[Formula (4.1)]{GidingGoncerzewicz2016:art}.

\noindent{\rm (iii)} For every $z \in D$, there holds\footnote{According to \eqref{eq:DefFBTW}, we define $\Gamma_v^+(z,\tau) := \sup\{y \in\RR: v(z,y,\tau) > 0 \}$.}
\[
\liminf_{\tau\to+\infty} \; \Gamma_v^+(z,\tau) - c_\ast\tau  \geq  \Gamma_\varphi^+(z) + \ell_0.
\]

\noindent{\rm (iv)} Finally, there holds
\[
\limsup_{\tau\to+\infty} \; \Gamma_v^+(z,\tau) - c_\ast\tau \leq \widetilde{\Gamma}_\varphi^+(z) + \ell_0,
\]
uniformly w.r.t $z \in D$, where $\widetilde{\Gamma}_\varphi^+$ denotes the concave envelope of $\Gamma_\varphi^+$.
\end{thm}
Also this statement deserves some additional comments. Part (i) completes the wave analysis by showing the uniqueness of the pair $(c_\ast,\varphi)$ and gives an explicit expression for the critical speed. The presence of the first eigenvalue of the Dirichlet Laplacian in $D$ is a consequence of the invariance principle mentioned above (see \cite[Theorem 3.2]{GidingGoncerzewicz2016:art}). From now on, in order to fix the notations, we will implicitly assume that $\varphi$ is the unique solution to \eqref{eq:Propphi}--\eqref{eq:WaveEqcast} satisfying $\sup\{\Gamma_\varphi^+(z) : z \in D\} = 0$.

Part (ii) proves that the solution to \eqref{eq:InDataTrasProb}--\eqref{eq:REACTIONTRANSFORMATION} converges (in absolute error) to the wave solution displaced by a suitable shift $\ell_0$ along the longitudinal direction. Remarkably, such shift is explicit in terms of the initial data (again this is a consequence of the invariance principle), but one has to note too that the convergence in absolute error gives no information about the solution near the boundary of the tube, since both the solution and the limit take the value 0 there. Our result not only provides a different approach to the study of the large time behaviour but also fixes this and, furthermore, provides a speed of convergence in compact sets.


Finally, parts (iii) and (iv) give a quite robust description of the long-time behaviour of the \emph{free boundary} of \eqref{eq:InDataTrasProb}--\eqref{eq:REACTIONTRANSFORMATION} in terms of the \emph{free boundary} of the wave solution. Notice that, if the set of $\{ \varphi > 0 \}$ is convex then (iii) and (iv) imply that $\Gamma_v^+(\cdot,\tau) - c_\ast\tau - \ell_0 \to \Gamma_\varphi^+(\cdot)$ uniformly on compact subsets of $D$, as $\tau \to +\infty$.

\subsection*{Our main result.} We can now state our main result.
\begin{thm}\label{thm:LongTimeBeh}
Let $c_{\ast} > 0$ be as in Theorem \ref{thm:ExTWsVaz} and let $v_0$ be defined as in \eqref{eq:InDataTrasProb}. Then there exists $t_0 > 0$ such that the solution $v$ to \eqref{eq:REACTIONTRANSFORMATION} with initial data $v_0$ satisfies the following two assertions:

\noindent{\rm (i)} For every $c \in (0,c_\ast)$,
\begin{equation}\label{eq:LTAInnerBeh}
\lim_{\tau \to +\infty} \sup_{z \in D, \; |y| \leq c\tau} \bigg|\frac{v(z,y,\tau)}{\Phi(z)} - 1 \bigg| = 0,
\end{equation}
where $\Phi$ is the unique nonnegative and nontrivial weak solution to \eqref{eq:STATIONARYPROBLEM}.

\noindent{\rm (ii)} For every $c > c_{\ast}$, there exists $\tau_c > 0$ such that
\begin{equation}\label{eq:LTAOuterBeh}
v(z,y,\tau) = 0 \quad \text{ in } \; D \times\{|y| \geq c\tau\}, \quad \forall \, \tau \geq \tau_c.
\end{equation}
Furthermore, there exist $C,T > 0$ depending on $m$, $D$ and the initial datum, such that for all $\tau \geq T$, the set $\Gamma_v^+(\tau) := \{(z,\Gamma_v^+(z,\tau)): z \in D \}$ is a locally H\"older\footnote{Actually, under suitable conditions on the initial data, $\Gamma_v^+(\tau)$ is locally smooth far away from $\partial\Omega$ (cf.  \cite{Caf1980:art,CaffVazWol1987:art,KienzlerVazquezKoch2018:art,Koch1990:art}).}
hypersurface, and the function $y = \Gamma_v^+(z,\tau)$ satisfies
\begin{equation}\label{eq:FBLongTimeWeak}
\lim_{\tau \to +\infty} \; \sup_{z \in D} \, \frac{\Gamma_v^+(z,\tau)}{\tau} = c_{\ast},
\end{equation}
and
\begin{equation}\label{eq:FBLongTimeStrong}
\sup_{\tau \geq T} \; \sup_{z \in D} \, |\Gamma_v^+(z,\tau) - \Gamma_\varphi^+(z) - c_{\ast}\tau| \leq C.
\end{equation}
\end{thm}
Before moving forward, it is important to rephrase the above statement and highlight a couple of significant facts. First of all, rewriting \eqref{eq:LTAInnerBeh} and \eqref{eq:LTAOuterBeh} in terms of the solution $u$ to \eqref{eq:PLETUBULARDOMAIN}, we obtain that
\[
\lim_{t \to +\infty} \sup_{z \in D, \; |y| \leq c\ln t} \bigg| \frac{u(z,y,t)}{U(z,t)} - 1 \bigg| = 0,
\]
for every fixed $c \in (0,c_{\ast})$, where $U = U(z,t)$ is the usual nonnegative solution to \eqref{eq:PLEBOUNDEDDOMAINSEPVARIABLE} in self-similar form \eqref{eq:FriendlyGiant}, while
\[
u(z,y,t) = 0 \quad \text{ in } \; D \times\{|y| \geq c\ln t\}, \quad \forall \, t \geq t_c,
\]
for every $c > c_{\ast}$ and some $t_c > 0$ large enough. The first relation is the most important and improves the long-time behaviour shown by V\'azquez from two perspectives: to the one hand, we have that $u$ converges to $U$ not only in absolute error, but also in \emph{relative error}; on the other, the convergence takes place not only on arbitrary compact subset of $\overline{\Omega}$, but also in \emph{expanding} sets of the form $D \times \{|y| \leq c \ln t\}$ for large times, where $c < c_\ast$ is fixed. In this sense, the \emph{Friendly Giant} is a \emph{stable steady state} to the PME in tubes and $c_\ast$ is the rate of convergence. Conversely, \eqref{eq:LTAOuterBeh} tells us that the \emph{free boundary} of $u$ must move slower than $c\ln t$ for large $t$'s, for every fixed $c > c_{\ast}$. We stress that \eqref{eq:LTAInnerBeh} and \eqref{eq:LTAOuterBeh} are \emph{sharp}, in the sense that both limits fail if $c = c_\ast$, in view of  Theorem \ref{thm:GildGoncRecap}, part (ii)).

To conclude, we notice that, as an immediate consequence of \eqref{eq:LTAInnerBeh} and \eqref{eq:LTAOuterBeh}, the limit in \eqref{eq:FBLongTimeWeak} follows, while the relation in \eqref{eq:FBLongTimeStrong} gives more precise information about the location of the \emph{free boundary} in the spirit of Gilding and Goncerzewicz, see Theorem \ref{thm:GildGoncRecap}, parts (iii)--(iv). Our estimate is rougher, but the bound from below is \emph{uniform} in $D$.
%
%
%
%
%
\subsection*{Structure of the paper.} The remaining part of the paper is divided in sections as follows.

In Section \ref{Section:RelErrorBdd} we improve the statement of \cite[Theorem 3.2]{Vazquez2004:art} by showing that the solution to problem \eqref{eq:InDataTrasProb}--\eqref{eq:REACTIONTRANSFORMATION} converges to the stationary solution $\Phi$ not only in absolute error, but also in \emph{relative error} in any compact subset of $\overline{\Omega}$.

In Section \ref{Section:ProofMainTheorem1} we show Theorem \ref{thm:LongTimeBeh} under the assumption that $D$ is star-shaped. In this setting the proof is more direct and is based on a comparison argument with a rescaled version of the TW  \eqref{eq:AnsatzTW}. The proof is independent from the one we give for general bounded domains.

In Section \ref{Section:ProofMainTheorem1bis} we treat the general case, by constructing new families of barriers modeled on the usual wave solution. To compare, we exploit Corollary \ref{cor:speed of convergence}, in which we quantify the rate of convergence in \emph{relative error} in compact subsets of $\overline{\Omega}$.

%
%
%
%
%
%
%
%
%
%
%
\section{Convergence in relative error in bounded cylinders}\label{Section:RelErrorBdd}
In this section we show convergence in relative error in sets of the form $D \times (-A,A)$, where $A > 0$ is arbitrarily fixed and, as consequence, in Corollary \ref{cor:speed of convergence}, we prove a second technical result, where we quantify the \emph{speed of convergence} in relative error.
\begin{prop}\label{LEMMACONVERGENCEINBOUDEDDOMAINS}
Let $v$ be the nonnegative weak solution to problem \eqref{eq:InDataTrasProb}--\eqref{eq:REACTIONTRANSFORMATION} with nontrivial initial data $v_0$. Then, for every $A > 0$,
\[
\lim_{\tau \to \infty} \sup_{z \in D, \, |y| \leq A}  \;\left| \frac{v(z,y,\tau)}{\Phi(z)} - 1 \right| = 0,
\]
where $\Phi = \Phi(z)$ is the unique nonnegative weak solution to the stationary problem \eqref{eq:STATIONARYPROBLEM}.
\end{prop}
\begin{proof} Let us fix $A > 0$ and $\vep \in (0,1)$. Our proof exploits a technical proposition from \cite{AronsonPeletier1985:art}: for this reason, we follow Aronson and Peletier notations and we prove the statement for the solution $u$ to \eqref{eq:PLETUBULARDOMAIN}, that is, we show that
\begin{equation}\label{eq:CCDoubleBound}
(1-\vep)U(z,t) \leq u(z,y,t) \leq U(z,t) \quad \text{ in } D\times(-A,A)\times (T,\infty),
\end{equation}
where $U = U(z,t)$ is the unique solution to \eqref{eq:PLEBOUNDEDDOMAINSEPVARIABLE} in self-similar form defined in
\eqref{eq:STATIONARYPROBLEM} and \eqref{eq:FriendlyGiant}, and $T \geq 0$ must be suitable chosen depending only on $m$, $D$, $u_0$, $A$ and $\vep$.

\noindent \emph{Step 1: The upper bound.} To show the upper bound in \eqref{eq:CCDoubleBound}, we follow the ideas of \cite[Proof of Theorem 1]{AronsonPeletier1985:art}. First we notice that, since $U$ is a weak solution to \eqref{eq:PLEBOUNDEDDOMAINSEPVARIABLE} and it is independent of $y$, then it is a weak solution to the same problem posed in $\Omega$. So, if there is $\overline{t} > 0$ such that
\begin{equation}\label{eq:CCBAbovet0}
u_0(z,y) \leq U(z,\overline{t}) \quad \text{ in } \Omega,
\end{equation}
then, by comparison,
\[
u(z,y,t) \leq U(z,\overline{t} + t) = (\overline{t} + t)^{-\frac{1}{m-1}} \Phi(z) \leq t^{-\frac{1}{m-1}} \Phi(z) = U(z,t) \quad \text{ in } \Omega\times(0,\infty),
\]
which is exactly our upper bound with $T = 0$ and $A = +\infty$. Assume by contradiction that~\eqref{eq:CCBAbovet0} does not hold for any $\overline{t} > 0$. Then there exists a sequence
$\{(z_k,y_k)\}_{k \in \NN} \subset \supp u_0 \subset \overline{\Omega}$ such that
\[
0 \leq \Phi^m(z_k) \leq \tfrac{1}{k} u_0^m(z_k,y_k),
\]
for every $k$. Since $\supp u_0$ is compact, there is $(\overline{z},\overline{y}) \in \supp u_0$ such that $(z_k,y_k) \to (\overline{z},\overline{y})$ as $k \to +\infty$ (up to passing to a suitable subsequence) and, since $u_0$ is bounded and $\Phi > 0$ in $D$, it must be $\overline{z} \in \partial D$. Consequently, dividing both sides of the above inequality by $\dist(z_k,\partial D)$ and using the third assumption in \eqref{eq:AssInitialData}, we pass to the limit as $k \to +\infty$ to find
\[
\partial_{\nu} \Phi^m(\overline{z}) = 0,
\]
where $\nu$ is the outward unit normal to $\partial D$, in contradiction with \cite[Proposition 1]{AronsonPeletier1985:art} (see also \eqref{eq:SSPosHopf}).

\noindent\emph{Step 2: A subsolution.} To prove the lower bound we construct an explicit subsolution and we exploit a technical result proved in  \cite{{AronsonPeletier1985:art}} in the bounded domain framework.

To construct the barrier from below, we fix $\alpha,\lambda > 0$ and consider the function
\[
f(y) := [ \lambda \cos(\alpha y) ]^{\frac{1}{m}}, \quad y \in I_\alpha:= (-\tfrac{\pi}{2\alpha},\tfrac{\pi}{2\alpha}),
\]
and we set
\begin{equation*}\label{eq:EllipticSubSol}
\Psi(z,y) := \Phi(z) f(y), \quad (z,y) \in \Omega_\alpha := D\times I_\alpha.
\end{equation*}
We claim that
\begin{equation}\label{eq:EllSubSolProb}
\begin{cases}
-\Delta \Psi^m \leq \tfrac{1}{m-1}\Psi \quad &\text{ in } \Omega_\alpha, \\
\Psi = 0                                \quad &\text{ in } \partial\Omega_\alpha,
\end{cases}
\end{equation}
for every $\lambda,\alpha > 0$ satisfying
\begin{equation}\label{eq:EllSubSolCondLamAlp}
\lambda^{\frac{m-1}{m}} \big[ 1 + \alpha^2(m-1) \|\Phi\|_{L^\infty(D)}^{m-1} \big] \leq 1.
\end{equation}
To see this, we notice that since $\nabla_{z,y} \Psi^m(z,y) = (f^m(y) \nabla_z \Phi^m(z), \Phi^m(z) (f^m)'(y))$ and $\Phi$ is a weak solution to \eqref{eq:STATIONARYPROBLEM}, it follows
\[
\begin{aligned}
-\Delta \Psi^m(z,y) &= - f^m(y) \Delta_z \Phi^m(z) - \Phi^m(z) (f^m)''(y) = \tfrac{\Phi(z)}{m-1} f^m(y) +\alpha^2\ \Phi^m(z) f^m(y) \\
&= \frac{\Phi(z)}{m-1} f^m(y)\big(1 +\alpha^2(m-1)\Phi^{m-1}(z) \big) \quad \text{ in } \Omega_\alpha.
\end{aligned}
\]
Consequently, $\Psi$ satisfies \eqref{eq:EllSubSolProb} if
\begin{equation}\label{eq:EllSubSolIneqf}
f^m(y)\big(1 +\alpha^2 (m-1)\Phi^{m-1}(z)\big) \leq f(y) \quad  \text{ in } \Omega_\alpha.
\end{equation}
Now, writing down $f$ explicitly and using that $m > 1$, we find
\[
f^m(y)\big(1 +\alpha^2 (m-1)\Phi^{m-1}(z)\big) \leq \lambda \cos(\alpha y) \big[ 1 + \alpha^2(m-1) \|\Phi\|_{L^\infty(D)}^{m-1} \big],
\]
and thus \eqref{eq:EllSubSolIneqf} is automatically satisfied if
\[
[ \lambda \cos(\alpha y) ]^{\frac{m-1}{m}} \big[ 1 + \alpha^2(m-1) \|\Phi\|_{L^\infty(D)}^{m-1} \big] \leq 1\quad \text{ in } \Omega_\alpha,
\]
which, in turn, holds true whenever $\lambda,\alpha > 0$ satisfy \eqref{eq:EllSubSolCondLamAlp}.

\noindent\emph{Step 3: The lower bound in thin cylinders.} In this step we prove the existence of $T > 0$ (depending only on $m$, $D$, $u_0$, and $\vep$) and $\dd > 0$ (depending only on $m$, $D$, and $\vep$) such that
\begin{equation}\label{eq:CCBBelowThinCyl}
u(z,y,t) \geq (1-\vep)U(z,t) \quad \text{ in } D\times(-\delta,\delta)\times(T,\infty).
\end{equation}
First we notice that, if $\lambda \leq 1$, then
\[
-\Delta (a \Psi^m) \leq \tfrac{a}{m-1}\Psi \leq \tfrac{a}{m-1}\|\Psi\|_{L^\infty(\Omega_\alpha)} \leq \tfrac{a}{m-1}\|\Phi\|_{L^\infty(D)} \leq 1 \quad \text{ in } \Omega_\alpha,
\]
for every $0 < a \leq (m-1)/\|\Phi\|_{L^\infty(D)}$, and thus, if $H$ denotes the unique solution to
\begin{equation}\label{eq:CCBarrierH}
\begin{cases}
-\Delta H = 1    \quad &\text{ in } \Omega_\alpha, \\
H = 0                    \quad &\text{ in } \partial\Omega_\alpha,
\end{cases}
\end{equation}
we obtain $-\Delta (a \Psi^m) \leq -\Delta H$ in $\Omega_\alpha$, with $a\Psi^m = H = 0$ in $\partial \Omega_\alpha$, and so
\begin{equation}\label{eq:HopfEllCompFund}
a\Psi^m \leq H \quad \text{ in } \Omega_\alpha,
\end{equation}
by virtue of the (elliptic) comparison principle.

Now, let $\underline{u}$ be the solution to problem \eqref{eq:PLEBOUNDEDDOMAIN} posed in $\Omega_\alpha$ (that is, $D = \Omega_\alpha$ in \eqref{eq:PLEBOUNDEDDOMAIN}). By comparison, we have $\underline{u} \leq u$ in $\Omega_\alpha \times (0,\infty)$. On the other hand, \cite[Proposition 5]{{AronsonPeletier1985:art}}\footnote{This result holds in domains with \emph{smooth} boundary: to bypass this technical difficulty, it is enough to solve problem \eqref{eq:CCBarrierH} in a bounded domain $\tilde{\Omega}$ with smooth boundary, satisfying $\Omega_\alpha \subset \tilde{\Omega} \subset \Omega$.}
assures the existence of $T_\ast,b > 0$ depending only on $m$, $u_0$, $D$ and $\alpha$ such that
\[
bH(z,y) \leq \underline{u}^m(z,y,T_\ast) \leq u^m(z,y,T_\ast) \quad \text{ in } \Omega_\alpha.
\]
Combining the above inequality with \eqref{eq:HopfEllCompFund}, we deduce $u^m(T_\ast) \geq ab \Psi^m$ in $\Omega_\alpha$, that is
\begin{equation}\label{eq:HopfInitialCompFund}
u(z,y,T_\ast) \geq (ab)^{\frac{1}{m}} \Psi(z,y) \quad \text{ in } \Omega_\alpha.
\end{equation}
Now, let $t_\ast > 0$ be fixed and consider the function
\[
w(z,y,t) := (t_\ast + t)^{-\tfrac{1}{m-1}} \Psi(z,y).
\]
Since $\Psi$ satisfies \eqref{eq:EllSubSolProb}, $w$ satisfies
\[
\begin{cases}
\partial_tw \leq \Delta w^m  \quad &\text{ in } \Omega_\alpha \times(0,\infty), \\
w = 0                                \quad &\text{ in } \partial\Omega_\alpha \times(0,\infty),
\end{cases}
\]
and thus, choosing $t_\ast > 0$ such that $(ab)^{\frac{1}{m}} \geq (t_\ast + T_\ast)^{-\tfrac{1}{m-1}}$, inequality \eqref{eq:HopfInitialCompFund} yields
\[
u(z,y,t) \geq (t_\ast + t)^{-\tfrac{1}{m-1}} \Psi(z,y) \quad \text{ in } \Omega_\alpha\times(T_\ast,\infty),
\]
by parabolic comparison. Now, take $\lambda := 1-\tfrac{\vep}{2}$ and notice that, by continuity, there exists $\dd_0 \in (0,1)$ depending only on $\vep$ and $m$, such that
\[
\cos(\alpha y) \geq (1-\tfrac{\vep}{2})^{m-1} \qquad \text{ if and only if } \qquad y \in (-\tfrac{\delta_0}{\alpha},\tfrac{\delta_0}{\alpha}).
\]
Consequently, we may choose $\alpha$ small enough (depending only $\vep$, $D$ and $m$) such that \eqref{eq:EllSubSolCondLamAlp} is satisfied and so, as a consequence of the last two inequalities, it follows that
\begin{equation}\label{eq:CCBBelowThinAst}
u(z,y,t) \geq (1-\tfrac{\vep}{2}) (t_\ast + t)^{-\tfrac{1}{m-1}} \Phi(z) \quad \text{ in } D\times (-\tfrac{\delta_0}{\alpha},\tfrac{\delta_0}{\alpha})\times(T_\ast,\infty).
\end{equation}
Choosing $T \geq T_\ast$ large enough (depending on $\vep$) and $\dd = \frac{\dd_0}{\alpha}$, we have
\[
u(z,y,t) \geq (1-\tfrac{\vep}{2}) \left( \frac{t}{t_\ast + t} \right)^{\tfrac{1}{m-1}} U(z,t) \geq (1-\vep) U(z,t) \quad \text{ in } D\times (-\delta,\delta)\times(T,\infty),
\]
and \eqref{eq:CCBBelowThinCyl} follows.

\noindent\emph{Step 4: The lower bound in finite cylinders.} To complete the proof of the lower bound, we fix $A > 0$ and we choose $\alpha > 0$ satisfying both \eqref{eq:EllSubSolCondLamAlp} and
\[
\dd = \frac{\dd_0}{\alpha} \geq A.
\]
Since $\dd_0$ depends only on $m$ and $\vep$, by \eqref{eq:EllSubSolCondLamAlp} we have that $\dd$ depends only on $m$, $D$, $\vep$ and $A$. At this point, we may repeat the argument above: using \cite[Proposition 5]{AronsonPeletier1985:art}, we find $T_\ast,b > 0$ (now depending also on $A$) for which \eqref{eq:HopfInitialCompFund} is satisfied, we choose $t_\ast > 0$ exactly as above, deducing that \eqref{eq:CCBBelowThinAst} holds true in $D \times (-A,A) \times (T_\ast,\infty)$. Finally, taking $T \geq T_\ast$ large enough (depending on $m$, $D$, $u_0$, $A$ and $\vep$), the lower bound in \eqref{eq:CCDoubleBound} follows.
\end{proof}
\begin{rem} Notice that Step 1 of the proof yields the existence of a time $\bar{t} > 0$ such that $u(z,y,t) \leq U(z,\bar{t}+t)$ in $\Omega\times(0,\infty)$. From now on, we will always assume that $t_0 \in (0,\bar{t})$, where $t_0$ is the parameter appearing in \eqref{eq:SECONDRENORMALIZATIONFORMULA}--\eqref{eq:InDataTrasProb}. In this way, we automatically have the universal upper-bound for the solution to \eqref{eq:REACTIONTRANSFORMATION}:
\[
v(z,y,\tau) \leq \Phi(z) \quad \text{ in } \Omega\times(\tau_0,\infty),
\]
where we recall that $\tau_0 := \ln t_0$.
\end{rem}

The last part of this section consists on proving the exponential speed of convergence of the solution $v$ to the friendly giant, a key result in order to stablish comparison on a parabolic boundary later on when comparing with the sub and supersolutions that converge to travelling waves. An intermediate result prior to this is the following lemma.

\begin{lem}\label{lem:concavity_mantained}
Given $A,k>0$, and $\varepsilon\in(0,1/2)$, let $f$ be the solution of
\begin{equation}\label{eq:SubSolExpConvFinal}
	\begin{cases}
		f_\tau = k(f^m)_{yy} + \frac{f}{m-1}(1-f^{m-1}),&(y,\tau)\in (-A,A)\times \RR_+,\\
		f(y,0)=1-\varepsilon,&y\in  [-A,A],\\
		f(\pm A,\tau)=1-\varepsilon,&\tau \in \RR_+.
	\end{cases}
\end{equation}

\noindent{\rm (i)} $(f^m)_{yy}\leq 0$ in $[-A,A]\times \RR_+$.

\noindent{\rm (ii)} There exists $\varepsilon_*=\varepsilon_*(m)\in(0,1/2)$ such that, if $\varepsilon \in (0,\varepsilon_*)$, then, given $B\geq0$, $B<A$,
\begin{equation}
\label{eq:exponential.convergence}
f(y,\tau)\geq 1-\delta e^{-\lambda\tau}\text{ for all }y\in[-B, B],\ \tau>T, \quad\text{for some constants }\delta,\lambda>0,\ T\ge0.
\end{equation}
\end{lem}
\begin{proof} (i) Let $h:=(f^m)_{yy}$. We want to check that $h \leq 0$ in $(-A,A)\times\RR_+$, i.e., $f^m$ is concave. Since $f \geq 1-\vep$ in $[-A,A]\times\RR_+$ (by comparison with the constant subsolution $1-\vep$), we have that $f$ is smooth. Consequently, we may change the order of differentiation to obtain
	$$
	h_\tau = \big((f^m)_{yy}\big)_\tau = \big((f^m)_\tau \big)_{yy}=\big( mf^{m-1}f_\tau \big)_{yy},
	$$
	and thus, using the equation of $f$,
	$$
	h_\tau = m\left[ kf^{m-1}h + \frac{1}{m-1}(f^m-f^{2m-1}) \right]_{yy}.
	$$
	Using that
	$$
	(f^{2m-1})_{yy}= \big[(f^m)^{2-1/m}\big]_{yy} = \left( 2-\tfrac{1}{m} \right)f^{m-1}h + \left( 2-\tfrac{1}{m}\right)\left( 1-\tfrac{1}{m}\right)\left((f^m)_y \right)^2f^{-1},
	$$
	we obtain the equation
	$$
	h_\tau = m k \left(f^{m-1}h\right)_{yy} + \tfrac{m}{m-1}\left[ 1-\left( 2-\tfrac{1}{m}\right)f^{m-1}\right]h   - \left(2-\tfrac{1}{m} \right) \left((f^m)_y \right)^2f^{-1}.
	$$
	Since $f\in (1-\varepsilon, 1)$,
the previous equation can be written in the form
	$$
	h_\tau = mk\left(a(t,y)h\right)_{y} + \tfrac{m}{m-1}b(t,y)h   -c(t,y),
	$$
	where $a$ is a $C^2$, bounded and positive function, $b$ is continuous and bounded and $c\geq 0$. As a consequence, we obtain by our initial assumption
	$$
	\begin{cases}
	\displaystyle h_\tau = mk\left(a(y,\tau)h\right)_{yy} + \frac{m}{m-1}b(y,\tau)h -c(y,\tau), &y\in (-A,A), \, \tau>0,\\
	h(y,0)=0, &y\in [-A,A],\\
	\displaystyle h(\pm A,\tau)=\frac{(1-\varepsilon)^m-(1-\varepsilon)}{m-1}<0,  &\tau>0,
	\end{cases}
	$$
	where we have used that $f$ satisfies the equation in~\eqref{lem:concavity_mantained} at the boundary $y=\pm A$ to compute the value of $h$ there.
Since $c \geq 0$ the function $\overline{h} = 0$ is a supersolution and so $h \leq 0$ follows by comparison.

\noindent (ii) The exponential speed of convergence of $ f(0,\tau)$ to 1 is a consequence of the ``flatness" of the solution near the value 1, which implies that the solution resembles that of the ordinary differential equation
$f' =  \frac{f}{m-1}(1-f^{m-1})$, which in turn can be approximated by its linearization around the level~1,
$f' = (1-f)$. However, we are loosing something in each approximation, and the rigorous proof requires to work a bit. Nevertheless, since it analogous to the one of \cite[Lemma 3.2]{Du-Quiros-Zhou:art}, we only sketch it.

We consider the change of variables $v=f^m$. Hence,  $v$ satisfies the equation
$$
\frac{1}{m}v^{1/m - 1}v_\tau = kv_{yy} + \tfrac{1}{m-1}(v^{1/m} - v).
$$
Notice that $(1-\varepsilon)^m\le v<1$. In particular, we can make $v$ arbitrarily close to 1, by taking $\varepsilon$ close to 0. As a consequence, if $\varepsilon$ is small enough (how small depending only on $m$),
$$
v^{1/m} - v \geq \frac{m-1}{2m}(1-v),
$$
so that $v$ is a supersolution to
\begin{equation}
\label{eq:v.supersolution}
\frac{1}{m}v^{1/m - 1}v_\tau= kv_{yy} + \frac{1}{2m}(1-v).
\end{equation}
Let $\psi$ be the solution to
$$
\begin{cases}
		\displaystyle\psi_\tau = \frac{km}{2}\psi_{yy} + \frac14(1-\psi),&(y,\tau)\in (-A,A)\times \RR_+,\\
		\psi(y,0)=1-\varepsilon,&y\in  [-A,A],\\
		\psi(\pm A,\tau)=1-\varepsilon,&\tau \in \RR_+.
	\end{cases}
$$
Solutions of this problem are known to converge exponentially to 1 in $[-B,B]$,
$$
\psi(y,\tau)\geq 1-\delta e^{-\lambda\tau}\text{ for all }y\in[-B, B],\ \tau>T, \quad\text{for some constants }\delta,\lambda>0,\ T\ge0,
$$
see  \cite[Lemma 2.6]{Du-Polacik}. Therefore, if we prove that $\psi$ is a subsolution to~\eqref{eq:v.supersolution} we are done.

On the one hand, $1-\varepsilon\le \psi<1$, so that $1\ge \psi^{\frac1m-1}/2$, since $\varepsilon\in(0,1/2)$.
On the other hand,  $g=\psi_\tau$ satisfies
$$
\begin{cases}
		\displaystyle g_\tau = \frac{km}{2}g_{yy} - \frac{g}4,&(y,\tau)\in (-A,A)\times \RR_+,\\
		\displaystyle g(y,0)=\frac\varepsilon4,&y\in  [-A,A],\\
		g(\pm A,\tau)=0,&\tau \in \RR_+,
	\end{cases}
$$
where we have used the equation satisfied by $\psi$ to compute the initial value of $g$. Hence, a direct comparison argument shows that $\psi_\tau\ge 0$. Combining both things,
$$
\displaystyle\frac{\psi^{\frac1m-1}}2\psi_\tau\le \psi_\tau = \frac{km}{2}\psi_{yy} + \frac14(1-\psi),
$$
from where we immediately  get, multiplying by $2/m$, that $\psi$ is a subsolution to~\eqref{eq:v.supersolution}, as desired.
\end{proof}
\begin{cor}\label{cor:speed of convergence}
Let $v$ be the nonnegative weak solution to problem \eqref{eq:InDataTrasProb}--\eqref{eq:REACTIONTRANSFORMATION}. Then there exist $\delta, \lambda > 0$ and $T\ge0$ such that
\[
\big(1-\delta e^{-\lambda(\tau-T)}\big)\Phi(z)\leq v(z, 0, \tau) \quad \text{ in } D\times [2T,\infty),
\]
where $\Phi = \Phi(z)$ is the unique nonnegative weak solution to the stationary problem \eqref{eq:STATIONARYPROBLEM}.
\end{cor}
\begin{proof}
Let us fix a domain $D\times (-A,A)\times [T,\infty)$, for $A,T>0$ that we will choose properly. To prove our statement, we compare the solution $v$ with a barrier from below defined as
$$
w(z,y,\tau):=f(y, \tau)\Phi(z),
$$
where $f$ satisfies \eqref{eq:SubSolExpConvFinal} with $k = \|\Phi\|_{L^\infty(D)}^{m-1}$. In order to do that, we need to check that $w$ is a subsolution of the problem, that $w$ and $v$ are ordered at the initial time and on the lateral boundary $\partial(D\times (-A,A))\times (T,\infty)$, and apply the comparison principle as above.

Given $\vep \in (0,1/2)$, Proposition \ref{LEMMACONVERGENCEINBOUDEDDOMAINS} implies that for every $A > 0$, there exists $T>0$, which may be assumed to be bigger than the value of $T$ in~\eqref{eq:exponential.convergence}, such that
$$
v(z,y,\tau)\geq (1-\varepsilon)\Phi(z) \quad \text{ in }    D \times [-A,A]\times [T,\infty).
$$
It is clear then, from Lemma \ref{lem:concavity_mantained}, that $w\leq v$ in the parabolic boundary of the problem. Next, we compute
$$
w_\tau - \Delta w^m - \frac{w}{m-1}= \Phi\left[f_\tau-\Phi^{m-1}(f^m)_{yy} -\frac{f-f^m}{m-1}\right].
$$
Take $k = \|\Phi\|_{L^\infty(D)}^{m-1}$. Since $(f^m)_{yy} \leq 0$, we have that
$$
-\Phi^{m-1}(f^m)_{yy}\leq -k(f^m)_{yy}
$$
and one readily finds that
$$
w_\tau - \Delta w^m - \frac{w}{m-1} \leq  \Phi\left[f_\tau-k(f^m)_{yy} -\frac{f-f^m}{m-1}\right]= 0
$$
and thus $w$ is a subsolution in $D\times [-A,A]\times [T,\infty)$. By comparison we obtain
$$
w(z,y,\tau-T)\leq v(z,y,\tau) \quad \text{ in }  D \times (-A,A)\times (T,\infty),
$$
and so our statement follows in light of \eqref{eq:exponential.convergence}.
\end{proof}

%
%
%
%
%

%
%
%
%
%
%
%
%
%
%
%
\section{Proof of Theorem \ref{thm:LongTimeBeh} when $D$ is star-shaped}\label{Section:ProofMainTheorem1}
We are ready to proceed with the proof of Theorem \ref{thm:LongTimeBeh}. We begin by assuming that the domain $D$ is star-shaped with pole at the origin $O$: thanks to this assumption, the dilation $\lambda D$ satisfies $D \subset \lambda D$ for $\lambda > 1$, and we can accordingly perturb the TW \eqref{eq:AnsatzTW} to construct a subsolution for \eqref{eq:InDataTrasProb}--\eqref{eq:REACTIONTRANSFORMATION}. We mention that this proof does not use Corollary \ref{cor:speed of convergence}, but only Proposition \ref{LEMMACONVERGENCEINBOUDEDDOMAINS}.

\begin{proof}[Proof of Theorem \ref{thm:LongTimeBeh} ($D$ star-shaped)] Assume that $D$ is star-shaped with pole $O$. Let $u_0$ be nontrivial satisfying \eqref{eq:AssInitialData}, $v_0$ satisfying \eqref{eq:InDataTrasProb}, $v = v(z,y,\tau)$ the corresponding solution to \eqref{eq:REACTIONTRANSFORMATION} and $c_\ast > 0$ as in Theorem \ref{thm:ExTWsVaz}.

We first prove the following claim: for every $\vep \in (0,1)$ and every $c \in (0,c_{\ast})$, there exists $T > 0$ depending only on $m$, $D$, $u_0$, $c$ and $\vep$, such that
\begin{equation}\label{eq:WaveConvLowerBound}
v(z,y,\tau) \geq (1-\vep) \Phi(z) \quad \text{ in } D \times [0,c\tau] \times [T,\infty).
\end{equation}
The proof of \eqref{eq:WaveConvLowerBound} is the hardest part and it is divided in several steps as follows.

\noindent\emph{Step 1.} Let us fix $\varepsilon \in (0,1)$ and define
\[
D_\varepsilon := \lambda_\vep D, \qquad   \lambda_\varepsilon := (1-\varepsilon)^{-\frac{m-1}{2}} > 1.
\]
As mentioned above, since $D$ is star-shaped, $D \subset D_{\varepsilon}$. Now, let $\Phi_{\varepsilon} = \Phi_{\varepsilon}(z)$ be the unique nonnegative weak solution to \eqref{eq:STATIONARYPROBLEM} posed $D_{\varepsilon}$ and let
\[
\tilde{\Phi}(z) := (1-\vep) \Phi_{\varepsilon}(\lambda_\vep z), \quad z \in D.
\]
Then, using the equation of $\Phi_\vep$ and the definition of $\lambda_\vep$, we have
\[
-\Delta \tilde{\Phi}^m(z) = - (1-\vep)^m \lambda_\vep^2 \Delta \Phi_\vep^m(\lambda_\vep z) = \tfrac{(1-\vep)^m \lambda_\vep^2}{m-1} \Phi_\vep(\lambda_\vep z) = \tfrac{(1-\vep)^{m-1} \lambda_\vep^2}{m-1} \tilde{\Phi}(z) = \tfrac{1}{m-1} \tilde{\Phi}(z) \quad \text{ in } D,
\]
and thus, since $\Phi = \tilde{\Phi} = 0$ in $\partial D$ (this follows since $\Phi_\vep = 0$ in $\partial D_\vep$), we obtain
\begin{equation}\label{eq:WaveConvPhiPhivep}
\Phi(z) = (1-\vep) \Phi_{\varepsilon}(\lambda_\vep z) \quad \text{ in } D,
\end{equation}
by uniqueness. Now, let us write
\[
1 - \varepsilon = \frac{1-a_{\varepsilon}}{1-b_{\varepsilon}}, \qquad  \quad  a_{\varepsilon} := \frac{\varepsilon(3-\varepsilon)}{2}, \quad b_{\varepsilon} = \frac{\vep}{2}.
\]
Notice that $a_{\varepsilon},b_{\varepsilon} \in (0,1)$ with $a_{\varepsilon},b_{\varepsilon} \to 0$ as $\varepsilon \to 0^+$ and, in addition, $a_{\varepsilon} \in (\vep,1)$ and $b_{\varepsilon} \in (0,a_{\varepsilon})$. In particular, as a consequence of \eqref{eq:WaveConvPhiPhivep}, it follows
\begin{equation}\label{eq:FUNDAMENTALINEQSTEADYSTATES}
\left(1 - b_{\varepsilon}\right) \Phi(z) = (1 - a_{\varepsilon}) \Phi_{\varepsilon}(\lambda_\varepsilon z), \quad z \in D.
\end{equation}
Further, by Proposition \ref{LEMMACONVERGENCEINBOUDEDDOMAINS} (with $A = 1$), there is $\tau_{\varepsilon} > 0$ (depending only on $m$, $D$, $u_0$ and $\vep$)  such that
\begin{equation}\label{eq:SECONDFUNDAMENTALINEQSTEADYSTATES}
v(z,y,\tau+\tau_{\varepsilon}) \geq \left(1 - b_{\varepsilon}\right)\Phi(z) \quad  \text{ in }   D\times[0,1)\times(0,\infty).
\end{equation}

\noindent\emph{Step 2.} To prove \eqref{eq:WaveConvLowerBound}, we build a subsolution of wave type. Our candidate is the function
\[
\psi(z,y,\tau) := (1 - a_\vep)\,\varphi_\vep(\lambda_\vep z,\lambda_\vep y - c_\vep\tau + \ell), \quad z \in D, \; y,\tau \geq 0, \; \ell \geq 0,
\]
where $\varphi_\vep = \varphi_\vep(z,\xi)$ is the nonnegative weak solution to
\[
\begin{cases}
\Delta \varphi_\vep^m + c_{\vep\ast}\partial_\xi \varphi_\vep + \frac{\varphi_\vep}{m-1} = 0  \quad &\text{ in } \Omega_\vep, \\
\varphi_\vep = 0 \quad &\text{ in } \partial\Omega_\vep,
\end{cases}
\]
satisfying
\begin{equation}\label{eq:WaveConvPropphivep}
\partial_\xi \varphi_\vep \leq 0,   \qquad    \lim_{\xi \to -\infty} \, \sup_{z\in D} \, \frac{\varphi_\vep(z,\xi)}{\Phi_\vep(z)} = 1,   \qquad   \varphi_\vep(z,\xi) = 0 \quad \text{for all }z\in D,\ \xi \geq \xi_\vep,
\end{equation}
for some $\xi_\vep \in \RR$ depending only on $m$, $D$ and $\vep$ (cf. Theorem \ref{thm:ExTWsVaz}). The speed $c_{\vep\ast} > 0$ is the unique number for which the above problem has a nontrivial weak solution, while $\Omega_{\varepsilon} := D_{\varepsilon} \times (0,\infty)$. The speed $c_\vep > 0$ is defined as
\begin{equation}\label{eq:WaveConvDefCvep}
c_\vep := c_\ast \left(\frac{1 - a_\vep}{1 - \varepsilon}\right)^{m-1}.
\end{equation}
Since $\varepsilon < a_\vep$, we have $c_\vep < c_\ast$ with $c_\vep \to c_{\ast}$ as $\varepsilon \to 0$ (this follows since $a_\vep \to 0$ as $\vep \to 0$). Let us fix $c \in (0,c_\ast)$. In light of the above observations and recalling that $\lambda_\vep \to 1$ as $\vep \to 0$, we deduce the existence of $\vep_0 \in (0,1)$ small enough depending on $c$, such that
\begin{equation}\label{eq:WavePropDefcvep}
c < \lambda_\vep c < c_\vep < c_{\ast}, \quad \forall \, \vep \in (0,\vep_0).
\end{equation}
From now on, we fix $\vep \in (0,\vep_0)$. Notice that, since $D \subset D_\vep$, we have $c_\ast < c_{\vep\ast}$ with $c_{\vep\ast} \to c_\ast$ as $\vep \to 0$ (cf.  \cite[Corollary 5.5]{Vazquez2007:art}\footnote{The speed $c_\ast$ is a monotone and continuous function of the domain.}) and so $\dd_\vep := c_{\vep\ast} - c_\ast > 0$ for every $\vep \in (0,1)$, with $\dd_\vep \to 0$ as $\vep \to 0$.

Now, since $\partial_\xi\varphi_\vep \leq 0$, we have by direct differentiation and \eqref{eq:WaveConvDefCvep}:
\begin{equation}\label{eq:WaveConvSubSolDt}
\begin{aligned}
\partial_\tau \psi &= -c_\vep(1-a_\vep) \partial_\xi \varphi_\vep = - c_\ast \frac{(1-a_\vep)^m}{(1-\vep)^{m-1}} \partial_\xi \varphi_\vep\\
& = -c_{\vep\ast} \frac{(1-a_\vep)^m}{(1-\vep)^{m-1}} \partial_\xi \varphi_\vep + \dd_\vep \frac{(1-a_\vep)^m}{(1-\vep)^{m-1}} \partial_\xi \varphi_\vep
&\leq -c_{\vep\ast} \frac{(1-a_\vep)^m}{(1-\vep)^{m-1}} \partial_\xi \varphi_\vep,
\end{aligned}
\end{equation}
where we have also used the definition of $\dd_\vep$. On the other hand, by definition of $\lambda_\vep$,
\[
\Delta \psi^m = (1-\vep)^m \lambda_\vep^2 \Delta \varphi_\vep^m = \frac{(1-a_\vep)^m}{(1-\vep)^{m-1}} \Delta \varphi_\vep^m,
\]
and thus, combining \eqref{eq:WaveConvSubSolDt} with the equation of $\varphi_\vep$ and recalling that $a_\vep > \vep$, it follows
\[
\begin{aligned}
\partial_\tau\psi - \Delta \psi^m - \tfrac{1}{m-1}\psi &\leq -c_{\vep\ast} \frac{(1-a_\vep)^m}{(1-\vep)^{m-1}} \partial_\xi \varphi_\vep - \frac{(1-a_\vep)^m}{(1-\vep)^{m-1}} \Delta \varphi_\vep^m - \tfrac{1}{m-1} (1-a_\vep) \varphi_\vep \\
&\leq - \frac{(1-a_\vep)^m}{(1-\vep)^{m-1}} \big( c_{\vep\ast} \partial_\xi \varphi_\vep + \Delta \varphi_\vep^m + \tfrac{1}{m-1} \varphi_\vep \big) = 0 \quad \text{ in } D \times [0,\infty) \times (0,\infty).
\end{aligned}
\]
Since $\psi = 0$ in $\partial D \times [0,\infty)\times(0,\infty)$ by construction, $\psi$ is a subsolution, as claimed at the beginning of this step.

\noindent\emph{Step 3.} On the other hand, we have
\[
(1 - a_\vep)\Phi_\vep(\lambda_\vep z) \geq \psi(z,y,\tau) \quad \text{ in } D \times [0,\infty)\times(0,\infty),
\]
by construction and the upper bound in \eqref{eq:CCDoubleBound} and thus, by \eqref{eq:FUNDAMENTALINEQSTEADYSTATES} and \eqref{eq:SECONDFUNDAMENTALINEQSTEADYSTATES}, we obtain
\begin{equation}\label{eq:THIRDFUNDAMENTALINEQSTEADYSTATES}
v(z,y,\tau+\tau_{\varepsilon}) \geq \psi(z,y,\tau) \quad  \text{ in }   D\times[0,1)\times(0,\infty).
\end{equation}
Further, since the support of $\varphi_\vep$ is bounded to the right, we may choose $\ell = \ell_{\varepsilon}$ large enough such that
\[
v(z,y,\tau_{\varepsilon}) \geq \psi(z,y,0) \quad \text{ in } D \times [0,\infty).
\]
Moreover, by \eqref{eq:THIRDFUNDAMENTALINEQSTEADYSTATES} we easily get $v(z,0,\tau+\tau_{\varepsilon}) \geq \underline{v}(z,0,\tau)$ for any $z \in D$, $\tau \geq 0$ and, since $\psi = 0$ in $\partial D \times [0,\infty)\times [0,\infty)$, we deduce by comparison
\[
v(z,y,\tau + \tau_\vep) \geq (1 - a_\vep)\,\varphi_\vep(\lambda_\vep z,\lambda_\vep y - c_\vep \tau + \ell_\vep) \quad \text{ in } D \times [0,\infty)\times [0,\infty).
\]
Now, in view of the first limit in \eqref{eq:WaveConvPropphivep}, we know that there is $\overline{\xi}_\vep \in \RR$ such that
\[
\varphi_\vep(\lambda_\vep z,\xi) \geq (1 - \tfrac{\vep}{2})\Phi_\vep(\lambda_\vep z) \quad \text{ in } D \times (-\infty,\overline{\xi}_\vep],
\]
and thus, combining the above two inequalities with \eqref{eq:FUNDAMENTALINEQSTEADYSTATES} and the definition of $b_\vep$, we obtain
\begin{equation}\label{eq:WaveConvAlmostLowBound}
v(z,y,\tau + \tau_\vep) \geq (1 - \tfrac{\vep}{2})(1 - a_\vep) \Phi_\vep(\lambda_\vep z) = (1 - \tfrac{\vep}{2})^2 \Phi(z) \geq (1-\vep)\Phi(z),
\end{equation}
for every $z \in D$ and every $y,\tau \geq 0$ satisfying $\lambda_\vep y - c_\vep \tau + \ell_\vep \leq \overline{\xi}_\vep$. Finally, by the second inequality in \eqref{eq:WavePropDefcvep}, whenever $y \leq c\tau$, we have
\[
\lambda_\vep y - c_\vep \tau + \ell_\vep \leq -(c_\vep - \lambda_\vep c) \tau + \ell_\vep \leq \overline{\xi}_\vep,
\]
for every $\tau \geq \overline{\tau}_\vep := \tfrac{\ell_\vep - \overline{\xi}_\vep}{c_\vep - \lambda_\vep c}$, and \eqref{eq:WaveConvLowerBound} follows by \eqref{eq:WaveConvAlmostLowBound}, choosing $T := \tau_\vep + \overline{\tau}_\vep$.

\noindent\emph{Step 4.} Now, we fix $c > c_\ast$ and prove \eqref{eq:LTAOuterBeh}. Let $\varphi$ the unique wave solution to \eqref{eq:REACTIONTRANSFORMATION} with speed $c_\ast$ satisfying \eqref{eq:Propphi} and let
\[
\overline{\psi}(z,y,\tau) = \varphi(z,y - c_{\ast}\tau - \ell), \quad z \in D, \; y,\tau \geq 0, \; \ell \geq 0.
\]
By \eqref{eq:CCBAbovet0}, we have
\[
v_0(z,y) := t_0^{\frac{1}{m-1}} u_0(z,y) \leq \bigg( \frac{t_0}{\overline{t}} \bigg)^{\frac{1}{m-1}} \Phi(z) \quad \text{ in } \Omega,
\]
for some $\overline{t} > 0$ depending only on $m$, $D$ and $u_0$. Thus, choosing $t_0 \leq 2^{-(m-1)}\overline{t}$, we obtain $v_0(z,y) \leq \tfrac{1}{2} \Phi(z)$ in $\Omega$. Consequently, in view of \eqref{eq:Propphi}, we may choose $\ell \geq 0$ large enough so that $\overline{\psi}(z,y,0) = \varphi(z,y - \ell) \geq v_0(z,y)$ in $\Omega$, and so the comparison principle yields
\[
\overline{\psi}(z,y,\tau) = \varphi(z,y - c_{\ast}\tau - \ell) \geq v(z,y,\tau) \quad \text{ in } \Omega \times(\tau_0,\infty),
\]
where $\tau_0 := \ln t_0$ according to \eqref{eq:InDataTrasProb} (notice that the comparison at the boundary $\partial\Omega\times(0,\infty)$ is immediate since $v = \overline{\psi} = 0$ there). Recalling that $\varphi(z,\xi) = 0$ in $D\times[\xi_0,\infty)$ for some $\xi_0 \in \RR$, we immediately see that
\begin{equation}\label{eq:WaveConvOuter}
0 \leq v(z,y,\tau) \leq \varphi(z,y - c_{\ast}\tau - \ell) = 0 \quad \text{ in } D \times [c\tau,\infty) \times [T,\infty),
\end{equation}
where $T := \tfrac{\xi_0 + \ell}{c - c_\ast} > 0$.

\noindent\emph{Step 5.} To conclude the proof, it is sufficient to repeat the above arguments working in $D \times (-\infty,0]$ and using the reflected wave solutions instead of $\varphi$ (cf. \eqref{eq:PropPhiReflected}): in the very same way, we obtain that \eqref{eq:WaveConvLowerBound} holds true in $D \times [-c\tau,0] \times [T,\infty)$ whenever $c < c_\ast$, while \eqref{eq:WaveConvOuter} in $D \times (-\infty,-c\tau] \times [T,\infty)$, for every $c > c_\ast$.
\end{proof}
%
%
%
%
%
%
%
%
%
%
%
%
\section{Proof of Theorem \ref{thm:LongTimeBeh} when $D$ is a bounded domain}\label{Section:ProofMainTheorem1bis}
In this section we present the proof of Theorem \ref{thm:LongTimeBeh} in its full generality, that is, dropping the assumption that $D$ is star-shaped. To do so, we construct new barriers by perturbing the TW solution in \eqref{eq:AnsatzTW}. Inspired by the work of Bir\'o \cite{Biro2002:art} (see also~\cite{Garriz2018:art}), the main idea is to look for solutions to \eqref{eq:REACTIONTRANSFORMATION} with form
\begin{equation}\label{eq:ApproxTW}
w(z,y,\tau) := f(\tau) \varphi(z,y-g(\tau)),
\end{equation}
where $\varphi$ is the wave profile satisfying \eqref{eq:WaveEqcast}, while $f:\RR_+ \to \RR_+$ and $g:\RR \to \RR$ are smooth functions satisfying $f(0) = f_0$, $g(0)=g_0$. The functions $f$ and $g$ will be suitably chosen to build a subsolution and a supersolution, respectively. In this perspective, we plug the \emph{ansatz} \eqref{eq:ApproxTW} into the equation in \eqref{eq:REACTIONTRANSFORMATION}: a direct computation and \eqref{eq:WaveEqcast} show that
\begin{equation}\label{eq:SignOperBarriers}
\begin{aligned}
\partial_\tau w - \Delta w^m - \tfrac{1}{m-1} w &= f'(\tau) \varphi - f(\tau)g'(\tau) \partial_\xi \varphi - f^m(\tau) \Delta \varphi^m - \tfrac{f(\tau)}{m-1} \varphi \\
&=
f'(\tau) \varphi - f(\tau)g'(\tau) \partial_\xi \varphi + c_\ast f^m(\tau) \partial_\xi \varphi + \tfrac{f^m(\tau)}{m-1} \varphi - \tfrac{f(\tau)}{m-1} \varphi \\
&= f(\tau) \partial_\xi \varphi \left[ c_\ast f^{m-1}(\tau) - g'(\tau) \right] + \varphi \left[ f'(\tau) + \tfrac{f^m(\tau)}{m-1} - \tfrac{f(\tau)}{m-1} \right].
\end{aligned}
\end{equation}
With this computation in mind, we proceed in two main steps. As already anticipated, we construct a subsolution and a supersolution using ODEs techniques, and then we use them to trap the solution to \eqref{eq:InDataTrasProb}--\eqref{eq:REACTIONTRANSFORMATION}. The crucial fact is that we can choose the functions $f$ and $g$ in such a way
\[
f(\tau) \sim 1, \quad g(\tau) - c_\ast \tau \sim c_0 \quad \text{ for } \tau \sim +\infty,
\]
for some constant $c_0$ depending only on $m$, $D$, $f_0$ and $g_0$. Thanks to these properties, the statement of Theorem \ref{thm:LongTimeBeh} will easily follow.

\begin{rem}\label{rem:ReflectedBarriers}
It is important to mention that the arguments below can be easily adapted to construct barriers with form \eqref{eq:ApproxTW} traveling towards the bottom-end of the tube by perturbing the reflected wave solution (cf. \eqref{eq:PropPhiReflected}) instead of $\varphi$.
\end{rem}
%
%
%
%
%
%
\subsection*{Construction of a subsolution.} To build a subsolution, we consider the solutions $f$ and $g$ of the following Cauchy problems
\begin{equation}\label{eq:SubSolSystem}
\begin{cases}
f' = \tfrac{f}{m-1}(1 - \delta(\tau) - f^{m-1}), \  &\tau >0, \\
f(0) = f_0
\end{cases}
\qquad\qquad
\begin{cases}
g' = c_\ast f^{m-1}, \ &\tau >0, \\
g(0) = g_0,
\end{cases}
\end{equation}
where $f_0 \in (0,1)$ and $g_0 \in \RR$ will be suitably chosen and
\begin{equation}\label{eq:Defdelta}
\delta(\tau):= \frac{2(m-1)\delta_0}{(1+\tau)^2},
\end{equation}
for some $\delta_0 \in (0,1)$.

\begin{lem}\label{lem:PropfSubsol}
Let $m > 1$, $f_0 \in (0,1)$ and let $f$ satisfying \eqref{eq:SubSolSystem}, with $\dd$ as in \eqref{eq:Defdelta}. Then, for every $\dd_0 > 0$, we have
\begin{equation}\label{eq:LiminftyfSub}
\lim_{\tau\to+\infty} f(\tau) = 1.
\end{equation}
Further, there exists $\overline{\dd} \in (0,1)$ depending only on $m$ and $f_0$ such that such that for every $\dd_0 \in (0,\overline{\dd})$
\begin{equation}\label{eq:BoundOnfSub}
f(\tau) \leq 1 - \frac{\dd(\tau)}{2(m-1)}, \quad \forall\tau \geq 0.
\end{equation}
\end{lem}
\begin{proof} We first notice that $0 < f(\tau) < 1$, for all $\tau > 0$. This follows from the comparison principle for first order ODEs: the assumptions $f_0 \in (0,1)$ and $\dd(\tau) > 0$ imply that the constants $0$ and $1$ are sub and supersolutions, respectively.

To check the validity of \eqref{eq:LiminftyfSub} it is enough to prove that the limit $\lim_{\tau \to +\infty}f(\tau) := \ell \in (0,1]$ exists (the fact that $\ell = 1$ follows by passing to the limit into the equation and using that $\dd(\tau) \to 0$ as $\tau \to +\infty$). With this goal in mind, we integrate the equation of $f$ to deduce
\begin{equation}\label{eq:ExLimfSub}
(m-1) \ln (f(\tau)/f_0) = \int_0^\tau 1 - f^{m-1}(s) \, \rd s - \int_0^\tau \dd(s) \, \rd s,
\end{equation}
and thus, since both the integral functions in the r.h.s. are non-decreasing and the second is bounded by definition of $\dd$, it follows that $\ell \in (0,1]$ exists.

To prove \eqref{eq:BoundOnfSub}, we construct a supersolution for $f$. We consider the function
\[
h(\tau) := 1 - \frac{\dd_0}{(1+\tau)^2} = 1 - \frac{\dd(\tau)}{2(m-1)},
\]
and we fix $\overline{\dd} \in (0,1)$ (depending only on $m$ and $f_0$) such that for every $\dd_0 \in (0,\overline{\dd}]$, we have $h(0) = 1 - \dd_0 \geq 1 - \overline{\dd} \geq f_0$ and
\[
(1 - \tilde{\dd})^{m-1} \geq 1 - 2(m-1)\tilde{\dd}, \quad \forall \tilde{\dd} \in (0,\dd_0].
\]
Then, using the above relation and the definition of $\dd$, we obtain
\[
\begin{aligned}
h' - \frac{h}{m-1}(1 - \delta(\tau) - h^{m-1}) &= 2\dd_0(1+\tau)^{-3} + \frac{h}{m-1}\Big\{ \dd(\tau) + [1- \dd_0(1+\tau)^{-2}]^{m-1} - 1 \Big\}  \\
& > \frac{h}{m-1}\Big\{ \dd(\tau) + 1 - 2(m-1)\dd_0(1+\tau)^{-2} - 1 \Big\} = 0,
\end{aligned}
\]
for all $\tau > 0$, that is, $h$ is a (strict) supersolution for $f$. We deduce $f(\tau) \leq h(\tau)$ for all $\tau \geq 0$ by comparison, that is \eqref{eq:BoundOnfSub}.
\end{proof}
\begin{rem} We remark that the control over the growth of $f$ in \eqref{eq:BoundOnfSub} is crucial in order to apply the comparison principle: to the one hand, we know that the solution $v$ to \eqref{eq:InDataTrasProb}--\eqref{eq:REACTIONTRANSFORMATION} grows exponentially fast in compact sets of $\Omega$ (cf. Corollary \ref{cor:speed of convergence}) while, on the other, the subsolution $w$ in \eqref{eq:ApproxTW} approaches the Friendly Giant $\Phi = \Phi(z)$ much slower. This will allow us to compare $w$ and $v$ on the ``vertical'' part of the boundary $D \times \{y=0\} \times (0,\infty)$.
\end{rem}
\begin{lem}\label{lem:CrucialLimitSubsol}
Let $m > 1$, $\dd_0 \in (0,1)$, $f_0 \in (0,1)$, $g_0 \in \RR$, and let $f,g$ as in \eqref{eq:SubSolSystem}. Then
\begin{equation}\label{eq:CrucialLimitSub}
\lim_{\tau\to+\infty} g(\tau) - c_\ast\tau = g_0 - 2c_\ast(m-1)\dd_0 + c_\ast(m-1) \ln f_0.
\end{equation}
\end{lem}
\begin{proof} Combining the equation of $g$ with \eqref{eq:ExLimfSub}, we easily see that
\[
\begin{aligned}
c_\ast(m-1) \ln (f(\tau)/f_0) &= c_\ast\tau - \int_0^\tau g'(s) \, \rd s - c_\ast \int_0^\tau \dd(s) \, \rd s\\
&= c_\ast\tau - g(\tau) + g_0 - 2c_\ast(m-1)\dd_0 \frac{\tau}{1+\tau},
\end{aligned}
\]
for all $\tau > 0$. Passing to the limit as $\tau \to +\infty$ and using \eqref{eq:LiminftyfSub}, \eqref{eq:CrucialLimitSub} follows.
\end{proof}
\begin{lem}\label{lem:BoundBelowSubsol}
Let $t_0 > 0$, $\tau_0 := \ln t_0$ and let $v$ be the weak solution to \eqref{eq:InDataTrasProb}--\eqref{eq:REACTIONTRANSFORMATION}. Then for every $f_0 \in(0,1)$ and $g_0 \in \RR$, there exist $\dd_0 \in (0,1)$ and $T > 0$ such that if $f$ and $g$ satisfy \eqref{eq:SubSolSystem}, then
\begin{equation}\label{eq:BoundBelowSubSol}
v(z,y,\tau + \tau_0 + T) \geq f(\tau)\varphi(z,y-g(\tau)) \quad \text{ in } \Omega\times(0,\infty).
\end{equation}
\end{lem}
\begin{proof} Let us fix $f_0 \in (0,1)$ and $g_0 \geq 0$ and choose $\delta_0 \in (0,1)$ such that
\[
\dd_0 < \min\big\{\overline{\dd},\tfrac{1}{2c_\ast(m-1)}\big\},
\]
where $\overline{\dd} \in (0,1)$ is as in Lemma \ref{lem:PropfSubsol} and set
\[
\underline{w}(z,y,\tau) := f(\tau) \varphi(z,y-g(\tau)).
\]
Using that $\varphi$ has bounded support to the right and that $f_0 \in (0,1)$, we may apply Proposition \ref{LEMMACONVERGENCEINBOUDEDDOMAINS} to deduce the existence of $T > 0$ such that
\[
\underline{w}(z,y,0) = f_0 \varphi(z,y-g_0) \leq v(z,y,\tau_0 + T) \quad \text{ in } D \times\RR_+.
\]
Furthermore, by \eqref{eq:BoundOnfSub} and Corollary \ref{cor:speed of convergence}, it follows
\[
\underline{w}(z,0,\tau) = f(\tau) \varphi(z,-g(\tau)) \leq \Big[ 1 - \frac{\dd(\tau)}{2(m-1)} \Big] \Phi(z) \leq v(z,0,\tau_0 + \tau + T) \quad \text{ in } D,
\]
for all $\tau \geq 0$, taking eventually $T > 0$ larger. Since $\underline{w} = v$ in $\partial \Omega \times (0,\infty)$ and $\underline{w}$ is a subsolution by construction (it is enough to combine \eqref{eq:SignOperBarriers} and \eqref{eq:SubSolSystem}), the inequality in \eqref{eq:BoundBelowSubSol} follows by the comparison principle.
\end{proof}
%
%
%
%
%
%
%
%
\subsection*{Construction of a supersolution.} To build a supersolution, we follow the argument above by considering solutions to the following Cauchy problems
\begin{equation}\label{eq:SuperSolSystem}
\begin{cases}
\bar{f}' = \tfrac{\bar{f}}{m-1}(1 - \bar{f}^{m-1}), \ &\tau >0, \\
\bar{f}(0) = \bar{f}_0
\end{cases}
\qquad\qquad
\begin{cases}
\bar{g}' = c_\ast \bar{f}^{m-1}, \  &\tau >0, \\
\bar{g}(0) = \bar{g}_0,
\end{cases}
\end{equation}
for some suitable choice of $\bar{f}_0 \in (0,1)$ and $\bar{g}_0 \in \RR$.

\begin{lem}\label{lem:PropfSupersol}
Let $m > 1$, $\bar{f}_0 \in (0,1)$, $\bar{g}_0 \in \RR$ and let $\bar{f}$ and $\bar{g}$ satisfying \eqref{eq:SuperSolSystem}. Then, $\bar{f}$ is increasing and
\begin{equation}\label{eq:LiminftyfSub1}
\lim_{\tau\to+\infty} \bar{f}(\tau) = 1.
\end{equation}
Further,
\begin{equation}\label{eq:CrucialLimitSuper}
\lim_{\tau\to+\infty} \bar{g}(\tau) - c_\ast\tau = \bar{g}_0 + c_\ast(m-1) \ln \bar{f}_0.
\end{equation}
\end{lem}
\begin{proof} The proof is similar to the one of Lemma \ref{lem:PropfSubsol}.
\end{proof}

\begin{lem}\label{lem:BoundAboveSupersol}
Let $t_0 > 0$, $\tau_0 := \ln t_0$ and let $v$ be the weak solution to \eqref{eq:InDataTrasProb}--\eqref{eq:REACTIONTRANSFORMATION}. Then, for every $\bar{f}_0 \in (0,1)$, there exist $\bar{g}_0 \in \RR$ and $t_0 > 0$ such that if $\bar{f}$ and $\bar{g}$ satisfy \eqref{eq:SubSolSystem}, then
\begin{equation}\label{eq:BoundAboveSuperSol}
v(z,y,\tau + \tau_0) \leq \bar{f}(\tau)\varphi(z,y-\bar{g}(\tau)) \quad \text{ in } \Omega\times(0,\infty),
\end{equation}
\end{lem}
\begin{proof} The proof follows the same idea of Lemma \ref{lem:BoundBelowSubsol}. Given $\bar{f}_0 \in (0,1)$ and setting
\[
\overline{w}(z,y,\tau) := \bar{f}(\tau) \varphi(z,y- \bar{g}(\tau)),
\]
we choose $\bar{g}_0 \in \RR$ such that $\supp v_0 \subset\subset \supp \overline{w}(\cdot,0)$. Then, we take $t_0 > 0$ small enough depending on $\bar{f}_0$ such that
\[
v_0(z,y) = t_0^{\frac{1}{m-1}} u_0(z,y) \leq \overline{w}(z,y,0) \quad \text{ in } \Omega.
\]
Recalling that $v = \overline{w} = 0$ on $\partial \Omega$ and noticing that $\overline{w}$ is a supersolution by construction, \eqref{eq:BoundAboveSuperSol} follows by comparison.
\end{proof}

We are ready to complete the proof of Theorem \ref{thm:LongTimeBeh}.

\begin{proof}[Proof of Theorem \ref{thm:LongTimeBeh}] Let us fix $\ve \in (0,1)$. By Lemma \ref{lem:BoundBelowSubsol} and Lemma \ref{lem:BoundAboveSupersol}, there are $\tau_0 < \tau_1$ such that
\begin{equation}\label{eq:DoubleBoundFinal}
f(\tau-\tau_0)\varphi(z,y-g(\tau-\tau_0)) \leq v(z,y,\tau) \leq \overline{f}(\tau - \tau_1)\varphi(z,y-\overline{g}(\tau - \tau_1)) \quad \text{ in } \Omega\times(0,\infty).
\end{equation}
Now, fix $c < c_\ast$ and assume $0 \leq y \leq c\tau$, $\tau > 0$. By \eqref{eq:LiminftyfSub}, \eqref{eq:CrucialLimitSub} and the first two relations in \eqref{eq:Propphi}, we have
\[
\begin{aligned}
v(z,y,\tau) &\geq f(\tau-\tau_0)\varphi(z,y-g(\tau-\tau_0)) \geq f(\tau-\tau_0)\varphi(z,c\tau - g(\tau-\tau_0)) \\
&\geq (1-\vep)^{1/2} \varphi(z,c\tau - g(\tau-\tau_0)) \geq (1-\ve)\Phi(z) \quad \text{ in } D\times[0,c\tau]
\end{aligned}
\]
for large $\tau$'s. Recalling that $\Phi$ is an upper-bound for $v$ and repeating the same argument when $-c\tau \leq y \leq 0$ with the reflected wave solutions (cf. Remark \ref{rem:ReflectedBarriers}), the limit in \eqref{eq:LTAInnerBeh} follows.

Now, fix $c > c_\ast$ and assume $y \geq c\tau$, $\tau > 0$. Then by \eqref{eq:LiminftyfSub1}, \eqref{eq:CrucialLimitSuper} and \eqref{eq:Propphi}, we obtain
\[
\begin{aligned}
v(z,y,\tau) &\leq \overline{f}(\tau-\tau_1)\varphi(z,y-\overline{g}(\tau-\tau_1)) \leq \varphi(z, c\tau - \overline{g}(\tau-\tau_0)) = 0  \quad \text{ in } D\times[c\tau,\infty),
\end{aligned}
\]
for large $\tau$'s. Exactly as above, we may repeat the same argument in the second half of the tube by Remark \ref{rem:ReflectedBarriers}, and the proof of \eqref{eq:LTAOuterBeh} is completed.

Finally, \eqref{eq:FBLongTimeWeak} follows by the arbitrariness of $c$ in \eqref{eq:LTAInnerBeh} and \eqref{eq:LTAOuterBeh}, whilst \eqref{eq:FBLongTimeStrong} by combining \eqref{eq:DoubleBoundFinal} with \eqref{eq:CrucialLimitSub} and \eqref{eq:CrucialLimitSuper}.
\end{proof}
%
%
%
%
%
%
%
%
%
%

%
%
%
%
%

%
%
%

\vskip .5cm


\noindent \textbf{\large \sc Acknowledgments.} A. Audrito has received funding from the European Union's Horizon 2020 research and innovation programme under the Marie Sk{\l}odowska--Curie grant agreement 892017 (LNLFB-Problems).

\noindent A. G\'arriz received finantial support from the ANR project DEEV ANR-20-CE40-0011-01.

\noindent F. Quir\'os was funded by MCIN/AEI (Spain) through projects MTM2017-87596-P, PID2020-116949GB-I00, RED2018-102650-T and the ICMAT-Severo Ochoa grant CEX2019-000904-S.

\



%
%
%

%
%
%

%

%
\end{document}